\documentclass[a4,12pt,reqno]{amsart}
\usepackage{amsmath,amsthm,amssymb}

\usepackage{color}

\theoremstyle{plain}
\newtheorem{thm}{Theorem}[section]
\newtheorem{lem}[thm]{Lemma}
\newtheorem{prop}[thm]{Proposition}

\newtheorem{df}[thm]{Definition}
\newtheorem{ex}[thm]{Example}

  \makeatletter
  \@addtoreset{equation}{section}
  \makeatother

\title[Schur functions and $A^{(2)}_{2}$]
{Schur function identities
arising from
the basic representation of $A^{(2)}_{2}$}
\author[Mizukawa, Nakajima, Seno and Yamada ]{Hiroshi Mizukawa
, Tatsuhiro Nakajima, Ryoji Seno \\and \\
Hiro-Fumi Yamada}
\thanks{The first  author was supported by KAKENHI 24740033. The last author was supported by KAKENHI 24540020. }
\address{Hiroshi Mizukawa, Department of Mathematics,  National Defense
Academy of Japan,  Yokosuka 239-8686, Japan}
\email{mzh@nda.ac.jp}
\address{Tatsuhiro Nakajima, Faculty of Economics, Meikai University, Urayasu 279-8550, Japan}
\email{tatsu.nkjm@gmail.com}
\address{Ryoji Seno, NEC System Technologies, Ltd. }
\email{sense.xt39@gmail.com}
\address{Hiro-Fumi Yamada, Department of Mathematics, Okayama University, Okayama 700-8530, 
Japan}
\email{yamada@math.okayama-u.ac.jp}
\date{}

\begin{document}
\maketitle
\begin{abstract}
A Lie theoretic interpretation is given for some formulas of Schur functions and
Schur $Q$-functions.  Two realizations of the basic representation of the Lie algebra
$A^{(2)}_2$ are considered; one is on the fermionic Fock space and the other is
on the bosonic polynomial space.  Via the boson-fermion correspondence,
simple relations of the vacuum expectation values of fermions turn out to be
algebraic relations of Schur functions.
\end{abstract}
\section{Introduction}
The aim of this paper is to clarify a Lie theoretic meaning of a famous formula
for rectangular Schur functions which is given in \cite{cgr} and \cite{my}.
We interpret the formula via an iterated action of the Chevelley generators
on a maximal weight vector.  From this point of view we can present 
a new formula as well.

In \cite{imny} we obtained certain algebraic formulas for Schur functions.  Those are derived naturally by looking at two realizations of the basic representation of the affine Lie algebra $A^{(1)}_{1} \cong D^{(2)}_{2}$. More precisely, the basic representation is realized both on the fermionic Fock space and the bosonic polynomial space.  These are connected by the ``boson-fermion correspondence'',
which takes the residue of the vertex operators.  Manipulations of fermions and vertex operators give certain relations for Schur functions.  In the present paper we apply this recipe to the case of the Lie algebra $A^{(2)}_{2}$.

\section{Preliminaries}


Let $\lambda$ be a strict partition. For each node $x \in \lambda$ 
in the $j$th column, we assign color $a(x)$ by the following 
rule:
\begin{center}
$a(x) = 
\begin{cases}
0 & (j \equiv 0,1 \pmod{3}),\\
1 &  (j \equiv 2  \pmod{3}).
\end{cases}.$
\end{center}
For example, the nodes of $\lambda = (5,4,2,1)$ are colored as 
$$
\begin{tabular}{ccccc} 
 0& 1& 0& 0& 1\\ 
 0& 1& 0& 0&  \\ 
 0& 1&  &  &  \\ 
 0&  &  &  &  \\ 
\end{tabular}.
$$

\noindent
Define 3-bar cores $c_{m} \,\,(m \in {\mathbb{Z}})$ by
\begin{equation*}
\begin{cases}
c_{m} = (3m-2, \ldots, 4, 1) & (m>0) ,   \\
c_{0} = \emptyset  & (m = 0) ,   \\
c_{-m} = (3m-1, \ldots, 5, 2) & (m>0) . 
\end{cases}
\end{equation*}
\noindent
Let
\begin{equation*}
I^{n}_{i}(c_{m}) = \{\mu\mid \mu\supset c_{m},|\mu| = |c_{m}| + n, 
a(x) =i \,\,(\forall x\in \mu/c_{m})\}
\end{equation*}
be the set of strict partitions obtained from $c_{m}$ by adding $i$-nodes $n$ 
times in succession. 
For example, 
$I^{2}_{0}(c_{-2}) = \left\{(7,2), (6,3), (6,2,1), (5,4), (5,3,1)\right\}$ and 
$I^{2}_{1}(c_{3}) = \left\{(8,5,1), (8,4,2), (7,5,2)\right\}$.

Let us recall the 3-bar quotient of a strict partition. 
Let $\lambda = (\lambda_{1}, \ldots, \lambda_{\ell})$ be a strict partition.
For $a = 0,1,2,$ set $\lambda^{(a)}$ as the subpartition of $\lambda$ consisting of parts congruent to $a$ mod 3.  
Put
$\lambda[0] = (\lambda_{i_{1}}/3, \ldots, \lambda_{i_{s}}/3)$, where
$\lambda^{(0)} = (\lambda_{i_{1}}, \ldots,  \lambda_{i_{s}})$.
Let $\ell_{1}$ (resp. $\ell_{2}$) be the length of $\lambda^{(1)}$ (resp. $\lambda^{(2)}$).
We fix an integer $k$ which is greater than or equal to $\frac{\lambda_{1}^{(2)}+1}{3}$.
We consider a descending sequence of integers:
$${\mathcal{M}}(k,\lambda)=(a_{1},a_{2},\ldots,a_{\ell_{1}},b_{1},\ldots,b_{k-\ell_{2}}),$$
where $a_{i}=\frac{\lambda^{(1)}_{i}-1}{3}$ and $(b_{1},\ldots,b_{k-\ell_{2}})$ is obtained 
by removing  $-\frac{\lambda_{i}^{(2)}+1}{3}\ (1 \leq i \leq \ell_{2})$ from $(-1,-2,\ldots,-k)$.
Then we define 
$$\lambda[1]={\mathcal{M}}(k,\lambda)-(\ell_{1}-\ell_{2}-1,\ell_{1}-\ell_{2}-2,\ldots,\ell_{1}-\ell_{2}-(k+\ell_{1}-\ell_{2})).$$
Since the only difference caused by the choice of $k$ is the length of a zero-string at the end,
$\lambda[1]$ is uniquely determined.
\begin{ex}
Take a partition 
$
\lambda = (11,9,8,4,3,2,1)
$.
We see that $\lambda[0] = (3, 1)$
and $\lambda^{(1)}=(4,1)$ and $\lambda^{(2)}=(11,8,2)$.
Put $k=5 \geq (11+1)/3$.
Then we obtain
\begin{equation*}
{\mathcal{M}}(k,\lambda)=(1,0,-2,-5).
\end{equation*}
We see that $\ell_{1}-\ell_{2}= -1$.
Finally we have $\lambda[1] = (1,0,2,-5)-(-2,-3,-4,-5)=(3,3,2,0)=(3,3,2)$.
We illustrate the above construction by so-called 3-bar abacus:
\begin{center}
\begin{tabular}{ccc} 
 $0$& $\textcircled{\scriptsize1}$& \textcircled{\scriptsize2}\\ 
 $\textcircled{\scriptsize3}$& \textcircled{\scriptsize4}& $5$\\ 
 $6$& $7$& \textcircled{\scriptsize8}\\ 
 $\textcircled{\scriptsize9}$& $10$& \textcircled{\scriptsize11}\\
  $12$& $13$& $14$
\end{tabular}
$\rightarrow$
\begin{tabular}{c|cc} 
 $0$& $\textcircled{\scriptsize0}$& \textcircled{\scriptsize-1}\\ 
 $\textcircled{\scriptsize1}$& \textcircled{\scriptsize1}& $-2$\\ 
 $2$& $2$& \textcircled{\scriptsize-3}\\ 
 $\textcircled{\scriptsize3}$& $3$& \textcircled{\scriptsize-4}\\
  $4$& $4$& $-5$
\end{tabular}
$\rightarrow$
\begin{tabular}{c|cc} 
 $0$& $\textcircled{\scriptsize0}$& $-1$\\ 
 $\textcircled{\scriptsize1}$& \textcircled{\scriptsize1}& $ \textcircled{\scriptsize-2}$\\ 
 $2$& $2$& $-3$\\ 
 $\textcircled{\scriptsize3}$& $3$& $-4$\\
  $4$& $4$& $ \textcircled{\scriptsize-5}$
\end{tabular}.
\end{center}
\end{ex}
For a strict partition $\lambda=(\lambda_{1},\lambda_{2},\ldots, \lambda_{2k})$
with $\lambda_{1}>\ldots \lambda_{2k}\geq 0$,
we define
\begin{equation*}
f(\lambda)= |\lambda^{(2)}|, \quad 
g(\lambda)=
\sum_{i \geq 1}
\#\{j\mid \lambda^{(1)}_{j}>\lambda^{(0)}_{i}\}, \quad 
h(\lambda) = \ell(\lambda^{(2)}).
\end{equation*}
\begin{df}
\begin{enumerate}
\item[$(1)$] For $\lambda \in I^{n}_{1}(c_{m})$ $(m\geq 0)$, we define
\begin{equation*}
\delta_{1}(\lambda)=(-1)^{f(\lambda)+\binom{n}{2}}.
\end{equation*}
\item[$(2)$] For $\lambda \in I^{n}_{0}(c_{-m})$ $(m\geq 0)$, we define
\begin{equation*}
\delta_{0}(\lambda)=
\begin{cases}
(-1)^{f(\lambda)+g(\lambda)} & (m:even),  \\
(-1)^{f(\lambda)+g(\lambda)+h(\lambda)}& (m:odd).
\end{cases}
\end{equation*}
\end{enumerate}
\end{df}
\begin{ex}
\begin{enumerate}
\item[$(1)$]
Let $\lambda=(11,8,4,2)\in I^{3}_{1}(c_{4})$. Then we have 
$f(\lambda)=11+8+2$ and $\delta_{1}(\lambda)=(-1)^{21+\binom{3}{2}}=1$.
\item[$(2)$]
Let $\lambda=(10,6,2)=(10,6,2,0)\in I^{3}_{0}(c_{-3})$, we have $f(\lambda)=2, g(\lambda)=2, h(\lambda)=1$
and $\delta_{0}(\lambda)=(-1)^{2+2+1}=-1$.
\end{enumerate}
\end{ex}
%
\section{Main result}
Let $t = (t_{1}, t_{2},t_{3} \ldots)$ and $s = (s_{1}, s_{3},s_{5} \ldots)$ be two families of indeterminates.  Define $h_{n}(t)$ by $\displaystyle{\exp \sum^{\infty}_{n=1}t_{n}z^{n}=\sum^{\infty}_{n=0}h_{n}(t)z^{n}}$.
The Schur function indexed by $\lambda$ 
is defined by
\begin{equation*}
S_{\lambda}(t) = \det(h_{\lambda_{i}+j-i}(t)).
\end{equation*}
Next define $q_{n}(s)$ by 
$\displaystyle{\exp\sum^{\infty}_{n=1}s_{2n-1}z^{2n-1}=\sum^{\infty}_{n=0}q_{n}(s)z^{n}}.$
For $m>n\geq 0$, we put
\begin{equation*}
Q_{m,n}(s)=q_{m}(s)q_{n}(s)+2\sum^{n}_{i=1}(-1)^{i}q_{m+i}(s)q_{n-i}(s).
\end{equation*}
If $m\leq n$ we define $Q_{m,n}(s)=-Q_{n,m}(s)$. Let $\lambda=(\lambda_{1},\ldots,\lambda_{2n})$ be a strict partition, 
where $\lambda_{1}>\cdots>\lambda_{2n}\geq 0$. Then the $2n\times2n$ matrix 
$(Q_{\lambda_{i},\lambda_{j}})$ 
is skew symmetric. Schur's $Q$-function $Q_{\lambda}$ is defined by
\begin{equation*}
Q_{\lambda}(s)={\rm pf} (Q_{\lambda_{i}, \lambda_{j}}).
\end{equation*}
\noindent
We denote by $\Box(m,n)$ the rectangular partition $(n^{m})$.
We now state our main result.
\begin{thm}\label{main}
For non-negative integers m and $n$, we have
\begin{enumerate}
\item[{\rm (1)}]
$\displaystyle{\sum _{\mu\in I^{n}_{1}(c_{m})}\delta_{1}(\mu)S_{\mu[1]}(t) = 
S_{\Box(m-n,n)}(2t^{(2)}) \label{311},}$
\item[{\rm (2)}]
$\displaystyle{
\sum _{\mu\in I^{n}_{0}(c_{-m})}\delta_{0}(\mu)Q_{\mu[0]}(s)S_{\mu[1]}(t) = 
\sum _{\mu\in I^{n}_{0}(c_{-m}),\mu[0] = \emptyset }\delta_{0}(\mu)S_{\mu[1]}(u)\label{312},}$
\end{enumerate}
where, in {\rm (1)}, $S_{\nu}(2t^{(2)}) = S_{\nu}(t)|_{t_{j}\mapsto 2t_{2j}}$ and, in {\rm (2)} ,
$
u_{j} = \begin{cases}
t_{j}  & (j: \text{even}), \\
t_{j}-s_{j} &  (j: \text{odd}) . 
\end{cases}
$
\end{thm}

The formula (1) is a well known plethysm identity (\cite{cgr,my}).

\begin{ex}

For $i=1$, we take $m=4$ and $n=2$. Then we have 
\begin{equation*}
S_{2^{4}}(t)-S_{32^{2}1}(t)+S_{3^{2}1^{2}}+S_{42^{2}}(t)-S_{431}+S_{4^{2}}=S_{\Box(2,2)}(2t^{(2)}).
\end{equation*}
For $i=0$, we take $m=2$ and $n=2$. Then we have
\begin{equation*}
S_{3}(t)-S_{21}(t)+Q_{1}(s)S_{1^{2}}(t)-Q_{2}(s)S_{1}(t)+Q_{21}(s)
=S_{3}(u)-S_{21}(u).
\end{equation*}
\end{ex}

\section{$A^{(2)}_{2}$ and fermions}
 
\noindent 
We consider the associative $\mathbb{C}$-algebra $\mathbb{B}$ defined by the generators 
$\beta_{n}(n\in \mathbb{Z})$ satisfing the anti-commutation relations:
\begin{equation*}
[\beta_{m},\beta_{n}]_{+} = \beta_{m}\beta_{n}+\beta_{n}\beta_{m} = (-1)^{m}\delta_{m+n,0}.
\end{equation*}
Let $\mathcal{F}$ be the Fock space which is a left $\mathbb{B}$-module generated by the vacuum $|{\rm vac}\rangle$  with 
$
\beta_{n}|{\rm vac}\rangle = 0\ (n<0).
$
Let $\mathcal{F}^{\dag}$ be the right $\mathbb{B}$-module 
generated by $\langle {\rm vac}|$ with
$
\langle {\rm vac}|\beta_{n} = 0 \ (n>0).
$
A general element of ${\mathcal F}$ or $\mathcal{F}^{\dag}$
is
called a state.
We define a bilinear pairing called {\it vacuum expectation value} by
\begin{equation*}
\mathcal{F}^{\dag}\otimes\mathcal{F}\rightarrow \mathbb{C},
\quad  \langle {\rm vac}|u\otimes v|{\rm vac}\rangle
\mapsto  \langle {\rm vac}|uv|{\rm vac}\rangle.
\end{equation*}
Here we put  $\langle {\rm vac}|{\rm vac}\rangle = 1$ and $\langle {\rm vac}|\beta_{0}|{\rm vac}\rangle = 0$.

\begin{df}
For a strict partition $\lambda = (\lambda_{1},\ldots,\lambda_{2k})$ 
with $\lambda_{1}>\ldots>\lambda_{2k}\geq 0$. Let $|\lambda\rangle$ denote the state
\begin{equation*}
|\lambda\rangle = \beta_{\lambda_{1}}\ldots\beta_{\lambda_{2k}}|{\rm vac}\rangle \in \mathcal{F}.
\end{equation*}
For $\lambda = \emptyset $, we set $|\lambda\rangle = |{\rm vac}\rangle$.
\end{df}

We discuss the basic representation of the affine Lie algebra of type $A_{2}^{(2)}=\langle e_{i},f_{i},h_{i}\mid i=0,1 \rangle$.
The concrete form of $f_{i}\ (i=0,1)$ on ${\mathcal F}$ is as follows (\cite{ny}).
\begin{equation}\label{lie}
f_{0} = \sqrt{2}\sum_{m\in \mathbb{Z}}(-1)^{m+1}\beta_{3m}\beta_{-3m+1},\qquad
f_{1} = \sum_{m\in \mathbb{Z}}(-1)^{m}\beta_{3m-1}\beta_{-3m+2}.
\end{equation}

Let $\mathcal{F}_{0}$ be the $A^{(2)}_{2}$-submodule of $\mathcal{F}$ generated by $|{\rm vac}\rangle$. 
Then $\mathcal{F}_{0}$ is isomorphic to the irreducible highest weight module $L(\Lambda_{0})$, where 
$\Lambda_{0}$ is the fundamental weight of $A_{2}^{(2)}$.
The weight system of $L(\Lambda_{0})$ is well known:
$$
P(\Lambda_{0})=\{\Lambda_{0}-p\delta+q\alpha_{1}\mid p\geq 2q^{2},p,q\in\frac{1}{2}\mathbb{Z},p+q\in\mathbb{Z}\},
$$
where  $\delta =2 \alpha_{0}+\alpha_{1}$ 
is the fundamental imaginary root of $A_{2}^{(2)}$.
A weight $\Lambda$ on the parabola $\Lambda_{0}-2q^{2}\delta+q\alpha_{1}$ is said to be maximal
in the sense that $\Lambda+\delta$ is no longer a weight. 
The following two propositions are easily verified.
\begin{prop}
Let $\lambda$ be a strict partition. 
\begin{enumerate}
\item[\rm{(1)}]\label{421}
Let $\mu$ be a partition obtained from $\lambda$ by replacing its part i by $i+1$ ($i>0$). 
Then we have
$$ (-1)^{i}\beta_{i+1}\beta_{-i}|\lambda\rangle =  
\begin{cases}
|\mu\rangle  &   i\ \text{is\ a\ part\ of}\  \lambda\ \text{and}\ i+1\ \text{ is\ not},\\
0 & otherwise.
\end{cases}
$$
\item
[\rm{(2)}]
Let $\mu$ be a partition obtained from $\lambda$ by adding a part 1.
Then we have
\begin{equation*}
\beta_{1}\beta_{0}|\lambda\rangle = 
\begin{cases}
\frac{1}{2}|\mu\rangle & \lambda_{\ell(\lambda)}\not=1\ \text{and}\ \ell(\lambda)\equiv 1 \pmod{2} ,\\
|\mu\rangle          & \lambda_{\ell(\lambda)}\not=1\ \text{and}\ \ell(\lambda) \equiv 0 \pmod{2}  , \\
0 & otherwise,
\end{cases}
\end{equation*}
\end{enumerate}
\end{prop}

\begin{prop}
A weight space of the maximal weight $\Lambda_{0}-2m^{2}\delta+m\alpha_{1}\ (m \in {\mathbb Z})$ is spanned  by $|c_{m}\rangle$.
\end{prop}
From these propositions we obtain 
the iterative action of $f_{i}\ (i=0,1)$ on $|c_{m}\rangle$ as follows.
\begin{prop}\label{p44}
Let $m$ be a non-negative integer.
\begin{enumerate}
\item[{\rm{(1)}}]
$\displaystyle{\frac{f^{n}_{1}}{n!}|c_{m}\rangle = 
2^{n}\sum_{\lambda\in I^{n}_{1}(c_{m})}|\lambda\rangle}$.
\item[{\rm{(2)}}]We put
$\varepsilon_{m} =
\begin{cases}
1 & (m :\text{odd}),\\
 0 &  (m :\text{even}).
 \end{cases}$
 Then we have
$\displaystyle{\frac{f^{n}_{0}}{n!}|c_{-m}\rangle = \sqrt{2}^{-\varepsilon_{m}}\sum_{\lambda\in I^{n}_{0}(c_{-m})}
\sqrt{2}^{a(\lambda)}|\lambda\rangle,}$
\end{enumerate}
where $a(\lambda) = \#\{j \mid \lambda_{j} \equiv 0 \pmod{2}\}$ for $\lambda = (\lambda_{1},\lambda_{2},\ldots,
\lambda_{\ell(\lambda)+\varepsilon_{\ell(\lambda)}})$.
\end{prop}
\begin{proof}
(1) is a
direct consequence of Proposition \ref{421} (1) . 
 For $\lambda\in I^{n}_{i}(c_{-m})$, write
$
\lambda-c_{-m} = (r_{1},r_{2},\ldots,r_{m},r_{m+1})
$ with $r_{m+1}=0$ or $1$.
Then the coefficient of $|\lambda\rangle$ is
\begin{equation}\label{3}
\frac{1}{2^{\varepsilon_{m}r_{m+1}}}\frac{\sqrt{2}^{n}}{n!}\frac{n!}{r_{1}!r_{2}!\cdots r_{m}!}.
\end{equation}
We compute
\begin{equation*}
r_{1}!r_{2}!\cdots r_{m}! = 2^{\#\{j \mid r_{j} = 2\}} = 
2^{(n-\#\{j\mid r_{j} = 1, j\leq m\}-r_{m+1})/2}.
\end{equation*}
If $m$ is even, then we have
$\displaystyle{
\#\{j \mid r_{j} = 1, j\leq m\} = 
\begin{cases}
a(\lambda)& (\lambda_{m+1} = 0),   \\
a(\lambda)-1 & (\lambda_{m+1} = 1). 
\end{cases}
}$
If $m$ is odd, then we have
$\displaystyle{
\#\{j \mid r_{j} = 1, j\leq m\} = 
\begin{cases}
a(\lambda)-1 & (\lambda_{m+1} = 0), \\
a(\lambda) & (\lambda_{m+1} = 1).
\end{cases}
}$
By applying these four results to (\ref{3}), we have the formula as desired.
\end{proof}
\section{Bosonization}
We rename $\beta_{n}$'s as follows.
\begin{equation}\label{3fer}
\phi_n=\beta_{3n} ,\quad \quad   \psi_{n}=\beta_{3n+1},\quad \quad   \psi^{*}_{n}=
(-1)^{-3n-1}\beta_{-3n-1}.
\end{equation}
They satisfy the anti-commutation relations:
\begin{align}
[\psi_{m}, \psi^{*}_{n}]_{+} = \delta_{m,n}, \; \; \, \quad \qquad \qquad [\psi_{m}, \psi_{n}]_{+} = 
[\psi^{*}_{m}, \psi^{*}_{n}]_{+} = 0,\label{anti1}\\
[\phi_{m},\phi_{n}\,]_{+} = (-1)^{m}\delta_{m+n,0},\quad \; \; \; [\psi_{m}, \, \phi_{n}]_{+} = [\psi^{*}_{m},\phi_{n}]_{+} = 0.\label{anti2}
\end{align}
Let us introduce the bosonic current operators
\begin{equation*}
H_{n} = \sum _{k\in \mathbb{Z}}:\psi_{k}\psi^{*}_{k+n}:,\quad 
{H'}_{2n+1} = \frac{1}{2}\sum _{k\in \mathbb{Z}}(-1)^{k+1}\phi_{k}\phi_{-k-(2n+1)},
\end{equation*}
where $:\beta_{n}\beta_{m}: $ denotes the normal ordering defined by
$:\beta_{n}\beta_{m}: = \beta_{n}\beta_{m}-\langle {\rm vac}|\beta_{n}\beta_{m}|{\rm vac}\rangle.$
These operators generate an infinite-dimensional Heisenberg algebra
\begin{equation*}
\mathfrak{H} =\bigoplus_{n\in \mathbb{Z},n\neq 0}
\mathbb{C}
H_{n}
\oplus \bigoplus_{n\in \mathbb{Z}}\mathbb{C}{H'}_{2n+1}\oplus\mathbb{C}c
\end{equation*}
where $c$ denotes the central element of $\mathfrak{H}$.
One has
\begin{equation*}
[{H'}_{2m+1}, {H'}_{2n+1}] = \frac{2m+1}{2}\delta_{m+n,0}\ c, \qquad 
[H_{m}, H_{n}] = \frac{m}{2}\delta_{m+n,0}\ c,\qquad  
[{H'}_{2m+1}, H_{n}] = 0.
\end{equation*}
We have a canonical $\mathfrak{H}$-module $S[\mathfrak{H}_{-}]$,
where $\mathfrak{H}_{-} = \bigoplus_{n<0}\mathbb{C}H_{n}\oplus\bigoplus_{m<0}\mathbb{C}{H'}_{2m+1}$.
Let $t_{n} = \frac{2}{n}H_{-n}, s_{2n-1} = \frac{2}{2n-1}{H'}_{-2n+1}$ for positive integers $n$.
Then we can identify $S[\mathfrak{H}_-]$ with the ring $\mathbb{C}[t,s]$ of polynomials of infinitely many variables
 $t=(t_{1},t_{2},t_{3}\ldots)$ and
$s=(s_{1},s_{3},s_{5},\ldots)$.
The action of $\mathfrak{H}$ on $\mathbb{C}[t,s]$ reads as follows.
\begin{align*}
H_{n}P(t,s) = \frac{\partial }{\partial t_{n}}P(t,s),\quad \quad \qquad \quad
H_{-n}P(t,s) = \frac{n}{2}t_nP(t,s),\\
{H'}_{2n-1}P(t,s) = \frac{\partial }{\partial s_{2n-1}}P(t,s),\qquad 
{H'}_{-2n+1}P(t,s) = \frac{2n-1}{2}s_{2n-1}P(t,s),
\end{align*}
where
$P(t,s)\in \mathbb{C}[t,s]$ and $n>0$.
The central element $c$ acts as identity.
We define the space of highest weight vectors with respect to $\mathfrak{H}$ by
\begin{equation*}
\Omega ={\rm Span}_{\mathbb{C}} \left\{ |v\rangle  \in \mathcal{F}\mid \:  H_{n}|v\rangle = 0,{H'}_{2n-1}|v\rangle = 0\ (\forall n>0)\right\}.
\end{equation*}
The space $\Omega$ has a basis $\{|\sigma,m\rangle \mid \: m\in \mathbb{Z}, \: \sigma = 0,1\}$, where
\begin{equation*}
|0,m\rangle =
\begin{cases}
\psi_{m-1}\cdots \psi_{0}|{\rm vac} \rangle\quad (m>0)  , &\quad  \\
|{\rm vac} \rangle  \quad \quad \quad \quad \; \, \; \; \,  \,  \; \: \,  (m=0),&\quad  \\
\psi^{*}_{m}\cdots \psi^{*}_{-1}|{\rm vac} \rangle\quad \:  \,    (m<0) , &\quad
\end{cases}
|1,m\rangle = 
\begin{cases}
\sqrt{2}\phi_{0}\psi_{m-1}\cdots \psi_{0}|{\rm vac} \rangle & (m>0) , \\
\sqrt{2}\phi_{0}|{\rm vac} \rangle  & (m=0) ,   \\
\sqrt{2}\phi_{0}\psi^{*}_{m}\cdots \psi^{*}_{-1}|{\rm vac} \rangle  &  (m<0) . 
\end{cases}
\end{equation*}
\begin{prop}\label{p51}
For a positive inter $m$, we have
$$
|c_{m}\rangle = (-1)^{m}{\sqrt{2}}^{-\varepsilon_{m}}|\varepsilon_{m},m\rangle,\ \ 
|c_{-m}\rangle = (-1)^{\binom{m}{2}+m}\sqrt{2}^{-\varepsilon_{m}}|\varepsilon_{m},-m\rangle.
$$
\end{prop}
Put $\alpha=\frac{1}{2}\alpha_{1}$ and consider the formal exponential $e^{\alpha}$.
Introducing a formal symbol $\theta$ which satisfies $\theta^2=1$, we define the space 
\begin{equation*}
Z = \bigoplus _{m\in \mathbb{Z},\sigma = 0,1}\mathbb{C}\theta^{\sigma}e^{m\alpha}.
\end{equation*}
The action of the Heisenberg  algebra $\mathfrak{H}$ naturally extends to the space $\mathbb{C}[t,s]\otimes_{\mathbb{C}} Z$. 

The boson-fermion correspondence is a 
canonical isomorphism 
\begin{equation*}
\Phi : \mathcal{F} \rightarrow \mathbb{C}[t,s]\otimes_{\mathbb{C}} Z
\end{equation*}
of $\mathfrak{H}$-modules,
which satisfies $\Phi(|\sigma,m\rangle) = \theta^{\sigma}e^{m\alpha}$ for $m\in \mathbb{Z}$ and $\sigma = 0,1$.
Introduce the states $\langle m,\sigma|
\in \mathcal{F}^{\dagger}\:  (m\in \mathbb{Z}, \; \sigma = 0,1)$ 
which are characterized by $\langle m,\sigma | \sigma^{\prime} ,n \rangle = 
\delta_{m,n} \delta_{\sigma,\sigma^{\prime}} \: 
(m,n \in \mathbb{Z},\; \sigma,\sigma^{\prime} = 0,1)$ 
and
\begin{equation*}
\langle m,\sigma |\phi_{n} = 0\quad (n>0),\quad \quad 
\langle m,\sigma |\psi_{n} = 0\quad (n>m),\quad \quad 
\langle m,\sigma |\psi^{*}_{n} = 0 \quad (n<m).
\end{equation*}
We define the Hamiltonian by $H(t,s) = \sum ^{\infty }_{n=1}t_{n}H_{n}+\sum^{\infty}_{n=1}s_{2n-1}{H'}_{2n-1}$.
Put $P_{m,\sigma}(t,s) = \langle m,\sigma|e^{H(t,s)}|v\rangle \in {\mathbb C}[t,s].$
Then we have $$\Phi(|v\rangle) = \sum _{m,\sigma}P_{m,\sigma}(t,s)\theta^{\sigma}e^{m\alpha}\ \ 
 ( |v\rangle \in 
\mathcal{F}).$$
The following formula is well known (\cite{imny}).
\begin{lem}\label{bs}
Let $j_{1}>\ldots >j_{a}\geq 0$ and $i_{1}>\ldots>i_{r}>m$. We have 
\begin{equation*}
\Phi(\phi_{j_{1}}\cdots\phi_{j_{a}}\psi_{i_{1}}\cdots \psi_{i_{r}}|0,m \rangle) = 
\sqrt{2}^{-a}Q_{(j_{1}\cdots j_{a})}(s)S_{(i_{1}-m,\ldots ,i_{r}-m)-\delta_{r}}(t)
\theta^{a}e^{(m+r)\alpha}
\end{equation*}
where $\delta_{r} = (r-1,r-2,\ldots 1,0)$.
\end{lem}
The following is the key formula for our main theorem.
\begin{prop}\label{zeta}
Let $m$ be a positive integer.
\begin{enumerate}
\item[$(1)$]For $\lambda$ $\in$ $I^{n}_{1}(c_{m})$,
there exists a sign $\zeta_{n,m,1}(\lambda)=\pm 1$ such that
\begin{equation*}
\Phi(|\lambda\rangle)=
\zeta_{n,m,1}(\lambda)\sqrt{2}^{-\varepsilon_{m}}S_{\lambda[1]}(t)\theta^{\varepsilon_{m}}e^{(m-2n)\alpha}.
\end{equation*}
\item[$(2)$]
For $\lambda$ $\in$ $I^{n}_{0}(c_{-m})$,  
there exists a sign $\zeta_{n,m,0}(\lambda)=\pm 1$ such that
\begin{equation*}
\Phi(\sqrt{2}^{a(\lambda)}|\lambda\rangle)=
\zeta_{n,m,0}(\lambda)Q_{\lambda[0]}(s)S_{\lambda[1]}(t)\theta^{\varepsilon_{n+m}}e^{(n-m)\alpha}.
\end{equation*}
\end{enumerate}
\end{prop}
\begin{proof}
Let 
 $\lambda=(\lambda_{1},\lambda_{2},\ldots,\lambda_{2k}) \in I_{i}^{n}(c_{\pm m})$.
From definition, we have $|{\rm vac} \rangle=\psi_{-m}\ldots \psi_{-1}|0,-m\rangle$.
By using 
 the anti-commutation relations
 (\ref{anti1}) and (\ref{anti2}), we have
$$|\lambda \rangle=\beta_{\lambda_{1}}\beta_{\lambda_{2}}\ldots \beta_{\lambda_{2k}}|{\rm vac}\rangle
=\pm \phi_{i_{1}}\ldots \phi_{i_{j}} \psi_{k_{1}}\ldots \psi_{k_{l}}|0,-m\rangle.$$
Here one observes that $(i_{1},\ldots,i_{j})=\lambda[0]$ and $(k_{1},\ldots,k_{l})={\mathcal{M}}(m,\lambda)$.
In particular we have $l=\ell(\lambda^{(1)})+(m-\ell(\lambda^{(2)}))$.
Then Lemma \ref{bs} gives the formula.
\end{proof}
\section{Vertex operators}
In this section we realize the action of $f_{i}$ on ${\mathbb C}[t,s]\otimes_{\mathbb C} Z$ in terms of vertex operators.
We introduce the formal generating functions
\begin{equation*}
\phi(z) = \sum_{n \in \mathbb{Z}}\phi_{n}z^{n}, \qquad 
\psi(z) = \sum_{n \in \mathbb{Z}}\psi_{n}z^{n}, \qquad 
\psi^{*}(z) = \sum_{n \in \mathbb{Z}}\psi^{*}_{-n}z^{n-1}. 
\end{equation*}
For $t = (t_{1}, t_{2}, t_{3}, \ldots)$, set
\begin{equation*}
\xi(t,z) = \sum^{\infty}_{n=1}t_{n}z^{n}, \qquad 
\xi_{1}(t,z) = \sum^{\infty}_{n=1}t_{2n-1}z^{2n-1}.
\end{equation*}
We define the multiplication operators $\theta$, $e^{\pm \alpha}$, and an operator 
$z^{\pm H_{0}}$ by 
\begin{equation*}
\theta(\theta^{n} e^{m\alpha}) = \theta^{n+1}e^{m\alpha},\ 
e^{\pm\alpha}(\theta^{n}e^{m\alpha}) = (-1)^{n}\theta^{n}e^{(m\pm 1)\alpha}\ {\rm and\ }
z^{\pm H_{0}}(\theta^{n}e^{m\alpha}) = z^{\pm m}(\theta^{n}e^{m\alpha}).
\end{equation*}
The following proposition is a direct consequence of the boson-fermion correspondence (\cite{imny}).  
\begin{prop}
We have
\begin{align*}
\Phi\phi(z)\Phi^{-1} &= \sqrt{2}^{-1}e^{\xi_{1}(s,z)}e^{-2\xi_{1}({\partial}_{s},z^{-1})}\theta,\\
\Phi\psi{(z)}\Phi^{-1} &= e^{\xi(t,z)} e^{-\xi({\partial}_{t},z^{-1})}e^{\alpha}z^{H_{0}},\quad \\
\Phi\psi^{*}(z)\Phi^{-1} &= e^{-\xi(t,z)}e^{\xi({\partial}_{t},z^{-1})}e^{-\alpha}z^{-H_{0}},
\end{align*}
where ${\partial}_{t} = (\frac{\partial}{\partial t_{1}},\frac{1}{2}\frac{\partial}{\partial t_{2}},
\frac{1}{3}\frac{\partial}{\partial t_{3}},\ldots)$. 
\end{prop}
\begin{lem}\label{63}
Let $V_{1}(z) = \Phi\psi^{*}(z)\psi^{*}(-z)\Phi^{-1}$ and $V_{0}(z) = \sqrt{2}\Phi\phi(-z)\psi(z)\Phi^{-1}$. Then we have
\begin{align*}
V_{1}(z) &= 2e^{-2\xi(t^{(2)},z^{2})}e^{\xi({\partial}^{(2)}_{t},z^{-2})}ze^{-2\alpha}z^{-H_{0}}(-z)^{-H_{0}},\\
V_{0}(z) &= e^{-\xi_{1}(s,z)+\xi(t,z)}\:e^{2\xi_{1}({\partial}_{s},z^{-1} )-\xi({\partial}_{t},z^{-1})}
\:\theta e^{\alpha}z^{H_{0}},
\end{align*}
where $t^{(2)} = (t_{2},t_{4},\ldots)$ and ${\partial}^{(2)}_{t} = 
(\frac{\partial}{\partial t_{2}},\frac{1}{2}\frac{\partial}{\partial t_{4}},
\frac{1}{3}\frac{\partial}{\partial t_{6}},\ldots)$.
\end{lem}
\begin{proof}
The expression of $V_{0}(z)$ is obvious. 
By Proposition 6.1, we have
\begin{eqnarray*}
V_{1}(z) &=& \Phi\psi^{*}(z)\psi^{*}(-z)\Phi^{-1}\\
 &=& e^{-\xi(t,z)}e^{\xi({\partial}_{t},z^{-1})}e^{-\xi(t,-z)}e^{\xi({\partial}_{t},-z^{-1})}
e^{-\alpha}z^{-H_{0}}e^{-\alpha}(-z)^{-H_{0}}.
\end{eqnarray*}
A standard calculus of vertex operators gives 
$e^{\xi({\partial}_{t},z^{-1})}e^{-\xi(t,-z)} = 2e^{-\xi(t,-z)}e^{\xi({\partial}_{t},z^{-1})}$
and
$z^{-H_{0}}e^{-\alpha} = ze^{-\alpha}z^{-H_{0}}.$
\end{proof}
Let define the formal contour integral 
$$\oint V(z) dz= {\rm Res}\ V(z)dz=V_{-1}$$
for $V(z)=\sum_{n \in {\mathbb Z}}V_{n}z^{n}$.
Summing up (\ref{lie}), (\ref{3fer}) and Lemma \ref{63} we can say that
\begin{equation*}
f_{1} = -\oint V_{1}(z)dz, \qquad f_{0} = -\oint z^{-1}V_{0}(z)dz.
\end{equation*}
We denote by $\Delta(z) = \det(z^{j-1}_{i})_{1\leq i<j\leq n}$ the Vandermonde determinant .
\begin{lem}\label{if1}
We have
\begin{align*}
V_{1}(z_{n})V_{1}(z_{n-1})\cdots V_{1}(z_{1})
= 2^{n}(z_{1}\cdots z_{n})\Delta(z^{2})^{2}&
e^{-2\sum^{n}_{i=1}\xi(t^{(2)},z^{2}_{i})}e^{\sum^{n}_{i=1}\xi({\partial }^{(2)}_{t},z^{-2}_{i})}\\
 &\times e^{-2n\alpha}z^{-H_{0}}_{n}(-z_{n})^{-H_{0}}\cdots z^{-H_{0}}_{1}(-z_{1})^{-H_{0}}.
\end{align*}
\end{lem}
\begin{proof}
We denote by $V^{+}_{1}(z), V^{-}_{1}(z)$ and $V^{0}_{1}(z)$, respectively, by
$e^{-2\xi(t^{(2)},z^{2})}, e^{\xi({\partial}^{(2)}_{t},z^{-2})}$ and 
 $ze^{-2\alpha}z^{-H_{0}}(-z)^{-H_{0}}$,
so that $V_{1}(z)=2V_{1}^{+}(z)V_{1}^{-}(z)V_{1}^{0}(z)$.
First we have 
\begin{equation*}
V^{0}_{1}(z_{n})\cdots V^{0}_{1}(z_{2})V^{0}_{1}(z_{1}) 
 = (z_{1}\cdots z_{n})\left(\prod ^{n}_{j=1}z^{4(j-1)}_{j}\right) 
e^{-2n\alpha}z^{-H_{0}}_{n}(-z_{n})^{-H_{0}}\cdots z^{-H_{0}}_{0}(-z_{1})^{-H_{0}}.
\end{equation*}
Second, for $i<j$, we have
\begin{equation*}
V^{-}_{1}(z_{j})V^{+}_{1}(z_{i}) = \left(1-\frac{z^{2}_{i}}{z^{2}_{j}}\right)^{2}V^{+}_{1}(z_{i})V^{-}_{1}(z_{j}).
\end{equation*}
\end{proof}
We recall the Schur polynomial indexed by a partition $\lambda=(\lambda_{1},\ldots,\lambda_{n})$
as a symmetric polynomial of $z=(z_{1},\ldots,z_{n})$:
\begin{equation*}
S^{\lambda}(z) = \det\left(z^{\lambda_{j}+j-1}_{i}\right)/\det\left(z^{j-1}_{i}\right).
\end{equation*}
The Schur polynomials satisfy the following orthogonality relation:
\begin{equation}\label{or}
(-1)^{\binom{n}{2}}\oint\cdots\oint S^{\lambda}(z)S^{\mu}(z^{-1})
\Delta(z)^{2}(z_{1}\cdots z_{n})^{-n}dz_{1}\cdots dz_{n}
 = n!\: \delta_{\lambda\mu},
\end{equation}
where $z^{-1}=(z_{1}^{-1},\ldots,z_{n}^{-1})$.
%
The Schur polynomial $S^{\lambda}(z)$ is equal to $S_{\lambda}(t)$ by putting 
$t_{j}=\frac{1}{j}(z_{1}^{j}+\ldots+z_{n}^{j})$.
This relation turns out to be the following Cauchy identity:  
\begin{equation*}\label{so}
e^{\sum^{n}_{i=1}\xi(t,z_{i})} = \sum_{\ell(\lambda)\leq n}S^{\lambda}(z)S_{\lambda}(t).
\end{equation*}
We remark
\begin{equation}\label{ci}
e^{-\sum^{n}_{i=1}\xi(t,z_{i})} = \sum_{\ell(\lambda)\leq n}(-1)^{|\lambda|}S^{\lambda}(z)S_{\lambda^{\prime}}(t),
\end{equation}
where $\lambda^{\prime}$ is the partition conjugate to $\lambda$.
\begin{prop}\label{p64}
For $m>0$, we have 
\begin{equation*}
\frac{f^{n}_{1}}{n!} \: \theta^{m}e^{m\alpha} = (-1)^{\binom{n}{2}}2^nS_{\Box(m-n,n)}
(2t^{(2)})\theta^{m}e^{(m-2n)\alpha}.
\end{equation*}
\end{prop}
\begin{proof}
By Lemma \ref{if1} we have
\begin{align*}
V_{1}(z_{n})V_{1}(z_{n-1})\cdots V_{1}(z_{1})\theta^{m}e^{m\alpha}
= (-1)^{mn}2^{n}\Delta(z^{2})^{2}
e^{-2\sum^{n}_{i=1}\xi(t^{(2)},z^{2}_{i})}(z_{1}\cdots z_{n})^{-2m+1}\theta^{m}e^{(m-2n)\alpha}.
\end{align*}
Therefore
\begin{align*}
f^{n}_{1}\theta^{m}e^{m\alpha} &= (-1)^{(m+1)n}2^{n}\oint\cdots\oint\Delta(z^{2})^{2}
e^{-2\sum^{n}_{i=1}\xi(t^{(2)},z^{2}_{i})}(z_{1}\cdots z_{n})^{-2m+1}
\:dz_{1}\cdots dz_{n}\:\theta^{m}e^{(m-2n)\alpha}\\
&=(-1)^{(m+1)n}2^n\oint\cdots\oint\Delta(w)^{2}
e^{-2\sum^{n}_{i=1}\xi(t^{(2)},w_{i})}(w_{1}\cdots w_{n})^{-m}
 dw_{1}\cdots dw_{n}\:\theta^{m}e^{(m-2n)\alpha}
\end{align*}
where $w_{i} = z^{2}_{i}$. 
Since
$
S^{\Box(n,m)}(w^{-1}) = (w_{1}w_{2}\cdots w_{n})^{-m},
$
we have, by (\ref{or}) and (\ref{ci}),
\begin{align*}
f^{n}_{1}\theta^{m}e^{m\alpha} = (-1)^{\binom{n}{2}}2^nn!\: S_{\Box(m-n,n)}(2t^{(2)})
\theta^{m}\, e^{(m-2n)\alpha}.
\end{align*}
\end{proof}
Next we consider the action of $f_{0}$.
\begin{lem}
We have
\begin{align*}
V_{0}(z_{n})V_{0}(z_{n-1})\cdots V_{0}(z_{1})
&= e^{-\sum^{n}_{i=1}\xi_{1}(s,z_{i})+\sum^{n}_{i=1}\xi(t,z_{i})}
e^{2\sum^{n}_{i=1}\xi_{1}({\partial}_{s},z^{-1}_{i})-\sum^{n}_{i=1}\xi({\partial}_{t},z^{-1}_{i})}\\
&\times(-1)^{\binom{n}{2}}\frac{\Delta(z)^{2}}{\prod _{1\leq i<j\leq n}(z_{i}+z_{j})}
\theta^{n}e^{n\alpha}(z_{1}\cdots z_{n})^{H_{0}}.
\end{align*}
\end{lem}
\begin{proof}
We denote by $V^{+}_{0}(z), V^{-}_{0}(z)$ and $V^{0}_{0}(z)$, respectively, by
$e^{-\xi_{1}(s,z)+\xi(t,z)}$, $e^{2\xi_{1}({\partial}_{s},z^{-1})-\xi({\partial}_{t},z^{-1})}$ and 
 $\theta e^{\alpha}z^{H_{0}}$,
so that $V_{0}(z)=V_{0}^{+}(z)V_{0}^{-}(z)V_{0}^{0}(z)$.
First we have
\begin{equation*}
V^{0}_{0}(z_{n})V^{0}_{0}(z_{n-1})\cdots V^{0}_{0}(z_{1}) = 
(-1)^{n(n-1)}\left(\prod_{1\leq i<j\leq n}
z^{i-1}_{j}\right)\;\theta^{n}e^{n\alpha}(z_{1}\cdots z_{n})^{H_{0}}.
\end{equation*}\\
Second, for $i<j$, we have
\begin{align*}
e^{2\xi_{1}({\partial}_{s},z^{-1}_{j})}e^{-\xi_{1}(s,z_{i})}
= \frac{z_{j}-z_{i}}{z_{j}+z_{i}} e^{-\xi_{1}(s,z_{i})} e^{2\xi_{1}({\partial}_{s},z^{-1}_{j})},\ 
e^{-\xi({\partial}_{t},z^{-1}_{j})}e^{\xi(t,z_{i})}
= \left(1-\frac{z_{i}}{z_{j}}\right)e^{\xi(t,z_{i})}e^{-\xi({\partial}_{t},z^{-1}_{j})}.
\end{align*}
Thus, for $i<j$, we obtain
\begin{equation*}
V^{-}_{0}(z_{j})V^{+}_{0}(z_{i}) = 
\frac{z_{j}-z_{i}}{z_{j}+z_{i}}\left(1-\frac{z_{i}}{z_{j}}\right)V^{+}_{0}(z_{i})V^{-}_{0}(z_{j})
\end{equation*}
as desired. 
\end{proof}
The computation above shows the integral formula.
\begin{prop}\label{int}
{\small
\begin{align*}
&\frac{f^{n}_{0}}{n!}\left(\theta^{m}e^{-m\alpha}\right) \\
&= \frac{(-1)^{\binom{n}{2}+n(m+1)}}{n!}\oint\cdots\oint
\frac{\Delta(z)^{2}}{\displaystyle{\prod _{1\leq i<j\leq n}}(z_{i}+z_{j})}
e^{-\sum^{n}_{i=1}\xi_{1}(s,z_{i})+\sum^{n}_{i=1}\xi(t,z_{i})}
\frac{dz_{1}\cdots dz_{n}} {(z_{1}\cdots z_{n})^{m+1}}\theta^{m+n}e^{(n-m)\alpha}\\
&= \frac{(-1)^{\binom{n}{2}+n(m+1)}}{n!}\oint\cdots\oint
\frac{\Delta(z)^{2}}{\displaystyle{\prod _{1\leq i<j\leq n}}(z_{i}+z_{j})}
e^{\sum^{n}_{i=1}\xi(u,z_{i})}
\frac{dz_{1}\cdots dz_{n}}{(z_{1}\cdots z_{n})^{m+1}}\theta^{m+n}e^{(n-m)\alpha},
\end{align*}
}
where 
$
u_{j} = 
\begin{cases}
t_{j} & (j: even), \\
t_{j}-s_{j} & (j: odd). 
\end{cases}
$
\end{prop}
\section{Determination of the signs}
In this section we will finalize our proof of the main theorem.
We have to determine the signs $\zeta_{n,m,i}(\lambda)\ (i=0,1)$ of Proposition \ref{zeta}.
To this end, we shall rewrite a state $|\lambda\rangle\ (\lambda \in I_{i}^n(c_{\pm m}),\ m>0)$ into its ``normal form" 
(appearing in the left hand side of Lemma \ref{bs}) along the following operations.
\begin{enumerate}
 \item[(M0)] Rename $\beta$'s into $\phi$'s, $\psi$'s and $\psi^{\ast}$'s.
 \item[(M1)] Rewrite the vacuum $|{\rm vac}\rangle=\psi_{-1}\psi_{-2}\cdots\psi_{-m}|0,-m\rangle$, 
 \item[(M2)] Move the leftmost $\psi^{*}_{j}$ to the (left) neighborhood of $\psi_{-j}$, and 
 remove $\psi^{*}_{j}$ by using the relation $\psi_{j}^{*}\psi_{j}=1-\psi_{j}\psi_{j}^{*}$.
 Do this process until all $\psi_{j}^{*}$'s disappear. 
 \item[(M3)] Move the leftmost $\phi_{j}$ to the (left) neighborhood of the leftmost $\psi_{k}$.
 Do this process until one gets the normal form.
\end{enumerate}
A sign is assigned to each (M$x$); $(-1)^{f(\lambda)}$ to (M0), $(-1)^{\varepsilon_{m}h(\lambda)}$ to (M2), 
and $(-1)^{g(\lambda)}$ to (M3). Note that, when $i=1$, the sign assigned to the sequence of (M1) and (M2)
equals  $(-1)^{\ell(\lambda^{(1)})+\ell(\lambda^{(2)})}=(-1)^{m}$.
\begin{prop}\label{signzeta}
For $\lambda \in I_{i}^n(c_{\pm m})\ (m>0)$, the sign $\zeta_{n,m,i}(\lambda)$ is given by
\begin{align*}
\zeta_{n,m,1}(\lambda)&=(-1)^{f(\lambda)+m}=(-1)^{m+\binom{n}{2}}\delta_{1}(\lambda),\\
\zeta_{n,m,0}(\lambda)&=(-1)^{f(\lambda)+g(\lambda)+\varepsilon_{m}h(\lambda)}=\delta_{0}(\lambda).
\end{align*}
\end{prop}
\begin{ex}Take
$\lambda = (20,18,16,12,8,7,2,0)\in I^{6}_{0}(c_{-7})$. Then $f(\lambda)=30,$ $g(\lambda)=3$ and $h(\lambda)=3$. 
\begin{align*}
|\lambda\rangle 
&=
 \beta_{20}\beta_{18}\beta_{16}\beta_{12}\beta_{8}\beta_{7}\beta_{2}\beta_{0}|{\rm vac}\rangle\\
&= 
(-1)^{20+8+2}\underline{\psi^{\ast}_{-7}}
\phi_{6}\psi_{5}\phi_{4}\psi^{*}_{-3}\psi_{2}\psi^{*}_{-1}\phi_{0}
\psi_{-1}\psi_{-2}\psi_{-3}\psi_{-4}\psi_{-5}\psi_{-6}\psi_{-7}|0,-7\rangle \\
&= 
(-1)^{30}(-1)
\phi_{6}\psi_{5}\phi_{4}\underline{\psi^{*}_{-3}}\psi_{2}\psi^{*}_{-1}\phi_{0}
\psi_{-1}\psi_{-2}\psi_{-3}\psi_{-4}\psi_{-5}\psi_{-6}|0,-7\rangle \\
&= 
(-1)^{30}(-1)^{1+1}
\phi_{6}\psi_{5}\phi_{4}\psi_{2}\underline{\psi^{*}_{-1}}\phi_{0}
\psi_{-1}\psi_{-2}\psi_{-4}\psi_{-5}\psi_{-6}|0,-7\rangle \\
&= 
(-1)^{30}(-1)^{1+1+1}
\phi_{6}\psi_{5}\phi_{4}\psi_{2}\phi_{0}
\psi_{-2}\psi_{-4}\psi_{-5}\psi_{-6}|0,-7\rangle \\
&= 
(-1)^{30+3}
\phi_{6}\psi_{5}\underline{\phi_{4}}\psi_{2}\phi_{0}
\psi_{-2}\psi_{-4}\psi_{-5}\psi_{-6}|0,-7\rangle \\
&= 
(-1)^{30+3}(-1)
\phi_{6}\phi_{4}\psi_{5}\psi_{2}\underline{\phi_{0}}
\psi_{-2}\psi_{-4}\psi_{-5}\psi_{-6}|0,-7\rangle \\
&= 
(-1)^{30+3}(-1)^{1+2}
\phi_{6}\phi_{4}\phi_{0}
\psi_{5}\psi_{2}
\psi_{-2}\psi_{-4}\psi_{-5}\psi_{-6}|0,-7\rangle.
\end{align*}
\end{ex}
Combining Propositions \ref{p44} (1), \ref{p51}, \ref{zeta} (1), \ref{p64} with \ref{signzeta}
one comes to Theorem \ref {main} (1).  In the formula 
\begin{align*}
&\sum_{\lambda \in I_{0}^{n}(c_{-m})}\delta_{0}(\lambda)Q_{\lambda[0]}(s)S_{\lambda[1]}(t)\\
&= \frac{(-1)^{\binom{n+m+1}{2}}}{n!}\oint\cdots\oint
\frac{\Delta(z)^{2}}{\displaystyle{\prod _{1\leq i<j\leq n}}(z_{i}+z_{j})}
e^{-\sum^{n}_{i=1}\xi_{1}(s,z_{i})+\sum^{n}_{i=1}\xi(t,z_{i})}
\frac{dz_{1}\cdots dz_{n}}{(z_{1}\cdots z_{n})^{m+1}}\\
&=\frac{(-1)^{\binom{n+m+1}{2}}}{n!}\oint\cdots\oint
\frac{\Delta(z)^{2}}{\displaystyle{\prod _{1\leq i<j\leq n}}(z_{i}+z_{j})}
e^{\sum^{n}_{i=1}\xi(u,z_{i})}
\frac{dz_{1}\cdots dz_{n}}{(z_{1}\cdots z_{n})^{m+1}},
\end{align*}
which comes from Proposition \ref{int}, one puts $s=0$ and $t=u$ to get  
$$\sum _{\mu\in I^{n}_{0}(c_{-m}),\mu[0] = \emptyset }\delta_{0}(\mu)S_{\mu[1]}(u).$$
This is nothing but Theorem \ref{main} (2).

We conclude the paper with a new formula of the
``trapezoidal" $Q$-functions.
 If we put $t_{2j}=u_{2j}=0$, then the right hand side of Theorem \ref{main}
(2) reduces to a single $Q$-function.
For non-negative integers $m$ and $n$ such that $m-n+1 \geq 0$,
we denote by $\Delta(m,n)$  the strict partition $(m,m-1,\ldots,m-(n-1))$
of length $n$, which we call a trapezoid.
\begin{thm}
Let $m$ and $n$ be non-negative integers such that $m-n+1\geq 0$. We have
$$Q_{\Delta(m,n)}(u)=(-1)^{\frac{(m+1)(m+2n)}{2}}\sum _{\mu\in I^{n}_{0}(c_{-m}),\mu[0] = \emptyset }\delta_{0}(\mu)S^{odd}_{\mu[1]}(u),$$
where $S^{odd}_{\lambda}(u)=S_{\lambda}(u)\big|_{u_{2j}\mapsto 0}$.
\end{thm}
\begin{proof}
We put $\Delta_{-1}(z)=\frac{{\prod _{1\leq i<j\leq n}}(z_{i}-z_{j})}{{\prod _{1\leq i<j\leq n}}(z_{i}+z_{j})}$.
The main theorem enables us to compute the following:
\begin{align*}
&(-1)^{\binom{n+m+1}{2}}\sum _{\mu\in I^{n}_{0}(c_{-m}),\mu[0] = \emptyset }\delta_{0}(\mu)S_{\mu[1]}(u)\Big|_{u_{2j}\mapsto 0}\\
&=\frac{1}{n!}
\oint\cdots\oint
\frac{\Delta(z)^{2}}{\displaystyle{\prod _{1\leq i<j\leq n}}(z_{i}+z_{j})}
e^{\sum^{n}_{i=1}\xi(u,z_{i})}
\frac{dz_{1}\cdots dz_{n}}{(z_{1}\cdots z_{n})^{m+1}}\Big|_{u_{2j}\mapsto 0}\\
&=\frac{1}{n!}
\oint\cdots\oint
{\Delta_{-1}(z)}
e^{\sum^{n}_{i=1}\xi_{1}(u,z_{i})}
\frac
{{(-1)^{\binom{n}{2}}}\Delta(z^{-1})}
{(z_{1}\cdots z_{n})^{m-n+1}}
\frac{dz_{1}\cdots dz_{n}}{z_{1}\cdots z_{n}}\\
&=\frac{{(-1)^{\binom{n}{2}}}}{n!}
\oint\cdots\oint
{\Delta_{-1}(z)}
e^{\sum^{n}_{i=1}\xi_{1}(u,z_{i})}
\sum_{\sigma\in S_{n}}{\rm{sgn}}(\sigma)
 z_{\sigma(1)}^{-m}z_{\sigma(2)}^{-m+1}\cdots z_{\sigma(n)}^{-m+n-1}
\frac{dz_{1}\cdots dz_{n}}{z_{1}\cdots z_{n}}\\
&=\frac{{(-1)^{\binom{n}{2}}}}{n!}
\sum_{\sigma\in S_{n}}{\rm{sgn}}(\sigma)
\oint\cdots\oint
{\Delta_{-1}(z)}
e^{\sum^{n}_{i=1}\xi_{1}(u,z_{i})}
 z_{\sigma(1)}^{-m}z_{\sigma(2)}^{-m+1}\cdots z_{\sigma(n)}^{-m+n-1}
\frac{dz_{1}\cdots dz_{n}}{z_{1}\cdots z_{n}}\\
&=\frac{{(-1)^{\binom{n}{2}}}}{n!}
\sum_{\sigma\in S_{n}}{\rm{sgn}}(\sigma)
Q_{\sigma(m,m-1,\ldots,m-n+1)}(u)\\
&={(-1)^{\binom{n}{2}}}Q_{\Delta(m,n)}(u).
\end{align*}
Here we use a vertex operator expression of $Q$-functions and 
$$Q_{\sigma(\lambda)}(z_{1},\ldots,z_{n}):=Q_{(\lambda_{\sigma(1)},\lambda_{\sigma(2)},\ldots,\lambda_{\sigma(n)})}(z_{1},\ldots,z_{n})={\rm sgn}(\sigma) Q_{\lambda}(z_{1},\ldots,z_{n})$$
 for $\sigma \in S_{n}$ and a strict partition $\lambda=(\lambda_{1},\lambda_{2},\ldots,
\lambda_{n})$ 
 (\cite{jing,mac}).
\end{proof}

\end{document}